\documentclass[reqno]{amsart}

\usepackage{tikz}
\usetikzlibrary{arrows,calc,decorations.pathreplacing,decorations.markings,intersections,shapes.geometric,through,fit,shapes.symbols,positioning,decorations.pathmorphing}

\usepackage{mystyle}
\usepackage[left=1.4in, right=1.4in]{geometry}
\usepackage{microtype}
\usepackage{lmodern}

\tikzset{%
  symbol/.style={
    draw=none,
    every to/.append style={
      edge node={node [sloped, allow upside down, auto=false]{$#1$}}
    },
  },
}

\numberwithin{thm}{subsection}
\numberwithin{equation}{subsection}

\title{Quantum Galois groups of subfactors}
\author[Bhattacharjee]{Suvrajit Bhattacharjee}
\address{Stat-Math Unit\\ Indian Statistical Institute\\ 203, B.T. Road\\ Kolkata-700108}
\email{suvra.bh@gmail.com}
\author[Chirvasitu]{Alexandru Chirvasitu}
\address{Department of Mathematics, University at Buffalo, Buffalo, NY 14260-2900, USA}
\email{achirvas@buffalo.edu}
\author[Goswami]{Debashish Goswami}
\address{Stat-Math Unit\\ Indian Statistical Institute\\ 203, B.T. Road\\ Kolkata-700108}
\email{goswamid@isical.ac.in}
\dedicatory{Dedicated to the memory of Prof. V.F.R. Jones}

\begin{document}

\begin{abstract}
  For a finite-index $\mathrm{II}_1$ subfactor $N \subset M$, we prove the existence of a universal Hopf $\ast$-algebra (or, a discrete quantum group in the analytic language) acting on $M$ in a trace-preserving fashion and fixing $N$ pointwise. We call this Hopf $\ast$-algebra the quantum Galois group for the subfactor and compute it in some examples of interest, notably for arbitrary irreducible finite-index depth-two subfactors. Along the way, we prove the existence of universal acting Hopf algebras for more general structures (tensors in enriched categories), in the spirit of recent work by Agore, Gordienko and Vercruysse.
\end{abstract}



\maketitle

%\tableofcontents

%\setcounter{tocdepth}{1}

\section*{Introduction}
The theory of subfactors is one of the cornerstones of the modern theory of operator algebras. Since Jones' beautiful and path-breaking discovery of a deep connection between subfactors and knot theory \cites{j-knts1,j-knts2}, there have been numerous applications of subfactors to diverse fields of mathematics and beyond (e.g. physics). Early on (e.g. \cite{ocneanu}), it was recognized that plain group theory is not sufficient to appropriately capture the symmetry of a subfactor. Additionally, the links with Hopf-algebra/quantum-group theory was conspicuous in the construction of several classes of subfactors from Hopf (or, more generally, weak Hopf) algebra (co)actions. As an example, one may recall the characterizations of depth-two finite index subfactors in terms of weak Hopf algebra actions (see for instance, \cite{nvcharacteri}).

The similarity between finite field extensions and finite-index subfactors suggests a Galois-theoretic approach to symmetry. Intense activity along these lines, following Ocneanu's pioneering work \cite{ocneanu}, has resulted in several thematically entwined constructions and insights. For a (necessarily incomplete) sampling, we mention
\begin{itemize}
\item Hayashi's face algebras and his version of Galois theory \cite{hayashiface},
\item Izumi-Longo-Popa's compact group version \cite{izumilongopopa}, or
\item Nikshych-Vainerman's weak Hopf algebra version \cite{nvgalois}.
\end{itemize}
In the purely algebraic context, Chase-Harrison-Sweedler's and Chase-Sweedler's study of Galois theory for commutative rings led to the notion of a Hopf-Galois extension (equivalent, in modern form, to the notion of a Jones triple \cite{nsw}). In fact, Sweedler's study of ``groups of algebras'' and that of Takeuchi's in the noncommutative setting led to the notion of a Hopf algebroid \cite{schauenberg} and finally to \cite{kadison}. The Galois-theoretic character of the crossed product construction, which is one of our main examples, is in fact a special case of a more general characterization due to Doi-Takeuchi of Hopf-Galois extensions with ``normal basis property'' (\cite[Theorem 9]{dt} or \cite[Theorem 8.2.4]{montgomery}).

A related approach to studying the symmetries of a subfactor $N \subset M$, also initiated in \cite{ocneanu} and explored in detail in subsequent work \cites{kawaquantgal, ekasymp, kawaflat} (see also the exposition in \cite[Chapters 9, 10 and 11]{ekquantsymm}), is to regard that symmetry as embodied in the four bimodule categories $\mathrm{Bimod}_{M-M}$, $\mathrm{Bimod}_{M-N}$, $\mathrm{Bimod}_{N-M}$ and $\mathrm{Bimod}_{N-N}$. This viewpoint relates back to the preceding remarks: such categories of bimodules are intimately related to Hopf algebroids \cite{schauenberg} via the Tannaka reconstruction theorem \cite{tannakaalgebrd}. For instance, the bimodule category $\mathrm{Bimod}_{N-N}$ with the obvious fibre functor should correspond to a Hopf algebroid over $N$ which in some sense is a universal symmetry object of the subfactor. This viewpoint is closer in spirit to Grothendieck's version of Galois theory, motivated by the theory of covering spaces and fundamental groups.

As Hopf algebroids and their actions/fixed point subalgebras are somewhat technical and not as familiar a subject as Hopf algebras and their representations/actions, we believe that it is perhaps useful to formulate a notion of subfactor symmetry in terms of Hopf-algebra actions. We do this below, describing these notions of symmetry in terms of universal properties in the spirit of Wang's work on compact quantum groups \cite{wangqsymm}. We defer introducing the relevant notation and definitions, but summarize the main results here. First, we prove that the requisite universal objects always exist (see Theorem \ref{th:ex}):

\begin{thmu}
  Let $N\subset M$ be an inclusion of finite factors, and equip $M$ with its canonical tracial state $\tau$. Then, there exist Hopf $*$-algebras
  \begin{itemize}
  \item $\QGal(N\subset M)$, acting on $M$ universally so as to fix $N$ pointwise, and
  \item $\QGal_{\t}(N\subset M)$, acting as above with the additional constraint that the action preserve the trace $\tau$.
  \end{itemize}
\end{thmu}

As for more explicit descriptions for Hopf algebras, we have 

\begin{thmu}
  Let $N\subset M$ be a finite-index subfactor. Then, $\QGal_{\t}(N\subset M)$ consists of those elements $h \in Q^*_{aut}(\en({}_N M_N),\t_1)$ such that \[h \cdot (xy)=(h_1\cdot x)(h_2\cdot y)\] for all $x,y \in M$.
\end{thmu}

For irreducible depth-2, finite-index, type-$\2$ subfactors we can give a fuller picture:

\begin{thmu}%\label{crossed}
  Let $M$ be a finite factor admitting an outer action by a finite dimensional Hopf $C^*$-algebra $H$. Then $\QGal(M\subset M\rtimes H)=H^*$.
\end{thmu}

% % we prove the existence of a universal Hopf $\*$-algebra $\QGal(N \subset M)$ for a $\2$ subfactor $N \subset M$ of finite index, which acts on $M$ such that the action preserves the canonical trace of $M$ and $N$ is in the fixed point subalgebra for the action. In more precise terms
% % 

% % Let us end this introduction by briefly describing the plan of this short article. In Section \ref{section-def-exist}, we recall some terminology from subfactor theory, state and prove Wang's results in the action language and finally give the definition of the quantum Galois group and prove its existence. Section \ref{section-examples} computes the quantum Galois groups of the inclusions coming from the crossed product construction. In the final Section \ref{section-epilogue}, we describe some of the directions which we are looking at presently and we hope to include these in future versions of this preliminary draft.
% %

\subsection*{Acknowledgments}
The authors are grateful to Prof. Y. Kawahigashi for pointing out the reference \cite{nvsurvey} to us, to Prof. Z. Liu for bringing to our attention \cite{liu} and its potential connections to the present work (we comment on this in more detail in Subsection \ref{subse:inv} below.). We also thank Prof. D. Penneys for pointing us to \cites{jp0,jp1}. And finally, the authors would like to thank the referee for her/his careful reading, suggestions for improvements and for the reference \cite{hopftensor}.

The first author is grateful to the third author and Indian Statistical Institute for offering him a visiting scientist position. The second author was partially supported through NSF grant DMS-2001128. The third author is partially supported by J.C. Bose fellowship.

%%%%%%%%%%%%%%%%%%%%%%%%%%%%%%%%%%%%%%%%%%%%%%%%%%%%%%%%%%%%%%%%%%%%%%%%%%%%%
%%%%%%%%%%%%%%%%%%%%%%%%%%%%%%%%%%%%%%%%%%%%%%%%%%%%%%%%%%%%%%%%%%%%%%%%%%%%%
\section{Preliminaries}\label{se.prel}

We write $k$ for a generic base field fixed throughout, though we frequently specialize to $k=\bC$. Algebras and morphisms between them are always unital, and similarly (or rather dually) for coalgebras.

For the background on coalgebra or Hopf algebra theory needed below the reader can consult any number of good sources, such as \cites{swe,montgomery,maj-qg,rad-bk} (we cite results more specifically when needed). Sweedler notation is in use throughout, with implied summation. So for an element $h\in H$, its image through the comultiplication $\Delta:H\to H\otimes H$ is denoted by
\begin{equation*}
  \Delta h = h_1\otimes h_2. 
\end{equation*}
We will make heavy use of duality in several guises. To summarize, for a Hopf algebra $H$ we denote
\begin{itemize}
\item by $H^*$ the full vector space dual of $H$; it is, in general, only an algebra \cite[\S 1.2]{montgomery}, and a Hopf algebra when $H$ is finite-dimensional \cite[Proposition 1.4.2]{maj-qg};
\item by $H^{\circ}$ the {\it finite dual} of $H$ \cite[Definition 1.2.3]{montgomery}; it is always a Hopf algebra \cite[\S 9.1]{montgomery};
\item by $\widehat{H}$ a Hopf algebra {\it in duality} with $H$ via a pairing
  \begin{equation*}
    \<-,-\> : \widehat{H} \tens H \to k
  \end{equation*}
  satisfying the usual compatibility conditions:
  \begin{equation*}
    \<ab,c\> = \<a,c_1\>\<b,c_2\>,
  \end{equation*}
  etc.; see \cite[Definition 1.4.3]{maj-qg} or \cite[Definition 7.7.6]{rad-bk}. 
\end{itemize}

We use standard notation for coalgebraic and/or Hopf-algebras structure: $\Delta:H\to H\otimes H$ for the comultiplication, $\varepsilon:H\to k$ for the counit, $S:H\to H$ for the antipode, etc.

For an action of a Hopf algebra $H$ on an algebra $A$ making the latter into a (left) {\it module algebra} \cite[Definition 4.1.1]{montgomery} we write $A\rtimes H$ for the {\it smash product} denoted by $A\sharp H$ in \cite[Definition 4.1.3]{montgomery}.

\begin{defn}
  Let $H$ be a Hopf algebra and $A$ a left $H$-module algebra. The {\it commutant} $A'\cap A\rtimes H$ (or just plain $A'$ when the context is understood) is the subalgebra of $A\rtimes H$ consisting of elements that commute with every $a\in A\subseteq A\rtimes H$.

  An action of $H$ on $A$ is
  \begin{itemize}
  \item {\bf outer} if the commutant $A'\cap A\rtimes H$ consists only of the scalars;
  \item {\bf minimal} if the commutant $(A^H)'\cap A$ similarly consists only of scalars. 
  \end{itemize}  
\end{defn}

In general, an inclusion $A \subset B$ is said to be {\bf irreducible} if the (relative) commutant $A' \cap B$ is reduced to the scalars. Thus an action of $H$ on $A$ is outer (minimal) if $A \subset A\rtimes H$ (respectively, $A^H \subset A$) is irreducible.

We will frequently work over the complex numbers, with complex $*$-algebras and $*$-{\it co}algebras (e.g. the $\circ$-coalgebras of \cite{ash}): (co)algebras equipped with involutive, conjugate-linear self-maps `$*$' that reverse the (co)multiplication. Hopf $*$-algebras (\cite[Definition 2.1]{vd-dual}) are
\begin{itemize}
\item complex Hopf algebras
\item as well as $*$-algebras
\item such that $\Delta$ is a $*$-algebra morphism
\item and $*\circ S$ (i.e. the map $x\mapsto S(x)^*$) is an involution.
\end{itemize}
A Hopf $*$-algebra will automatically be a $*$-coalgebra under this involution (see e.g. the discussion in \cite[\S 1]{chi-fou}). For Hopf $*$-algebras $H$ and $H$-module $*$-algebras $A$ we will typically also assume compatibility between the action and the $*$-structures:
\begin{equation*}
  (h\cdot a)^* = S(h)^*\cdot a^*,\ \forall h\in H,\ a\in A. 
\end{equation*}

Recall also, for future reference, the following notion generalizing that of a module algebra over a Hopf algebra (e.g. \cite[Chapter VII]{swe}). 

\begin{defn}
  A linear morphism $\triangleright:C\otimes A\to B$ for two unital algebras $A$, $B$ and a (counital) coalgebra $C$ {\it measures $A$ to $B$} (or simply {\it measures} if $A$, $B$, etc. are understood) if
  \begin{itemize}
  \item We have
    \begin{equation*}
      c\triangleright (aa') = (c_1\triangleright a)(c_2\triangleright a'),\ \forall c\in C,\ a,a'\in A
    \end{equation*}
    (where $c\mapsto c_1\otimes c_2$ is the usual Sweedler notation for the comultiplication), and
  \item $c\triangleright 1 = \varepsilon(c)1$ for all $c\in C$, where the two $1$s are the units of $A$ and $B$ respectively.
  \end{itemize}
  When $C$ is a complex $*$-coalgebra and $A$ and $B$ are $*$-algebras we will typically impose the additional compatibility condition
  \begin{equation*}
    (c\triangleright a)^* = c^*\triangleright a^*. 
  \end{equation*}
  for $c\in C$ and $a\in A$.
\end{defn}

\begin{rem}
  Equivalently, the $*$-preservation condition can be phrased as requiring that the algebra morphism
  \begin{equation*}
    A\to \mathrm{Hom}(C, B)
  \end{equation*}
  corresponding to $C\otimes A\to B$ via the hom-tensor adjunction be one of $*$-algebras, noting that a $*$-algebra structure on $B$ and a $*$-coalgebra structure on $C$ jointly make $\mathrm{Hom}(C,B)$ into a $*$-algebra:
  \begin{equation*}
    f^*(c) = f(c^*)^*,\ \forall f\in \mathrm{Hom}(C,B),\ c\in C,
  \end{equation*}
  where the two stars on the right are those on $C$ and $B$.
\end{rem}

\begin{comm}
  A word of caution: when working with $*$-coalgebras or Hopf $*$-algebras, the finite duals $H^{\circ}$ are understood to be compatible with the $*$-structure. This means that in defining $H^{\circ}$ we consider $*$-ideals rather than just ideals (\cite[Definition 1.2.3]{montgomery}), etc.  
\end{comm}

%%%%%%%%%%%%%%%%%%%%%%%%%%%%%%%%%%%%%%%%%%%%%%%%%%%%%%%%%%%%%%%%%%%%%%%%%%%%%
%%%%%%%%%%%%%%%%%%%%%%%%%%%%%%%%%%%%%%%%%%%%%%%%%%%%%%%%%%%%%%%%%%%%%%%%%%%%%
\section{Universal measuring Hopf algebras}\label{se:univ}

We will have to generalize the setup recalled in Section \ref{se.prel} somewhat, to account for algebras, coalgebras and Hopf algebras {\it internal} to the category of $*$-vector spaces (see subsection \ref{subse:enr} below). For that reason, we take a category-theoretic detour.

%%%%%%%%%%%%%%%%%%%%%%%%%%%%%%%%%%%%%%%%%%%%%%%%%%%%%%%%%%%%%%%%%%%%%%%%%%%%%
\subsection{Categories, enrichment, etc.}\label{subse:enr}

$\cV$ will always denote a ``sufficiently rich'' category to enrich over: symmetric, monoidal, (co)complete, closed in the sense of \cite[\S VII.7]{macl} (i.e. the functors $x\otimes-$ are left adjoints), and, for good measure, locally presentable \cite[Definition 1.17]{ar}. $\cV$ will furthermore meet the technical requirements of \cite[Prposition 47 and Theorem 54 1.]{prst-ft1}; while we do not list these in detail, we note that they are met by every one of the few concrete examples of relevance here:
\begin{itemize}
\item vector spaces over a fixed ground field;
\item complex $*$-vector spaces (see \ref{def:*vect} below).
\end{itemize}

The conditions required in \cite[Proposition 47 and Theorem 54 1.]{prst-ft1} ensure, via these results, that the various universal constructions used in the proofs below are available:

\begin{itemize}
\item The forgetful functor $\mathrm{Coalg}_{\cV}\to \cV$ from coalgebras in $\cV$ to $\cV$ has a right adjoint (the {\it cofree} coalgebra on an object in $\cV$).
\item The category $\mathrm{Coalg}_{\cV}$ is complete and cocomplete, and in fact locally presentable.
\item The inclusion functor
  \begin{equation*}
    \mathrm{HAlg}_{\cV}\to \mathrm{Bialg}_{\cV}
  \end{equation*}
  from $\cV$-Hopf algebras to $\cV$-bialgebras has a right adjoint: every bialgebra $B\in \cV$ admits a universal map $H\to B$ from a Hopf algebra $H\in \cV$.
\end{itemize}

Other examples of appropriate $\cV$ would be sets, vector spaces, graded vector spaces, representations over an algebraic group, etc. Of special importance to us will be

\begin{defn}\label{def:*vect}
  The category $\mathrm{Vect}^*_{\bC}$ of {\it $*$-vector spaces} consists of complex vector spaces $V$ equipped with involutive conjugate-linear self-maps $*:V\to \overline{V}$, where $\overline{V}$ denotes the vector space complex-conjugate to $V$ (so this notation indicates that `$*$' is conjugate-linear as a self-map of $V$, but {\it linear} as a map $V\to \overline{V}$).

  We identify
  \begin{equation}\label{eq:rev}
    \overline{V\otimes W}\cong \overline{W}\otimes \overline{V},
  \end{equation}
  so the tensor product of $*$-vector spaces $(V,*)$ and $(W,*)$ is again a $*$-vector space $(V\otimes W,*)$, with this latter $*$ map {\it reversing} tensorands when regarded as a linear map
  \begin{equation*}
    V\otimes W\to \overline{V\otimes W}\cong \overline{W}\otimes \overline{V} .
  \end{equation*}
\end{defn}

\begin{rem}
  In particular, because of the tensor-product reversal in \ref{eq:rev}, an algebra internal to $\mathrm{Vect}^*_{\bC}$ is nothing but a $*$-algebra in the usual sense: a complex algebra equipped with an anti-multiplicative, conjugate-linear involution.

  Similarly, a coalgebra in $\mathrm{Vect}^*_{\bC}$ is a $*$-coalgebra, a Hopf algebra therein is a Hopf $*$-algebra, etc.
\end{rem}

We will also work with categories $\cC$ {\it enriched} and {\it tensored} over $\cV$: `enriched' means a {\it $\cV$-category} in the sense of \cite[\S 1.2]{kel}, while `tensored' means that for every object $c\in \cC$ the functor
\begin{equation*}
  \cC\ni d\mapsto \hom_{\cC}(c,d)\in \cV
\end{equation*}
has a left adjoint; we denote that adjoint by $-\otimes c$, so expressions like
\begin{equation*}
  v\otimes c\text{ for } v\in \cV,\ c\in \cC
\end{equation*}
become legal (and we might, on occasion, reverse the order of the tensorands).

%%%%%%%%%%%%%%%%%%%%%%%%%%%%%%%%%%%%%%%%%%%%%%%%%%%%%%%%%%%%%%%%%%%%%%%%%%%%%
\subsection{A general framework for universal measuring}\label{subse:univ}

\cite[Theorem 7.0.4]{swe} shows that for any two algebras $A$ and $B$ there is a universal measuring coalgebra $M(A,B)$ equipped with a measuring morphism
\begin{equation*}
  M(A,B)\otimes A\to B,
\end{equation*}
where universality, here and in the sequel, means that any measuring $C\otimes A\to B$ factors as
\begin{equation*}
  \begin{tikzpicture}[auto,baseline=(current  bounding  box.center)]
    \path[anchor=base] 
    (0,0) node (l) {$C\otimes A$}
    +(3,.5) node (u) {$M(A,B)\otimes A$}
    +(6,0) node (r) {$B$}
    ;
    \draw[->] (l) to[bend left=6] node[pos=.5,auto] {$\scriptstyle \psi\otimes\id$} (u);
    \draw[->] (u) to[bend left=6] node[pos=.5,auto] {$\scriptstyle $} (r);
    \draw[->] (l) to[bend right=6] node[pos=.5,auto,swap] {$\scriptstyle $} (r);
  \end{tikzpicture}
\end{equation*}
for a unique (coalgebra, in this case) morphism $\psi:C\to M(A,B)$. Furthermore, \cite[Exercise immediately preceding \S7.1]{swe} shows that $M(A,A)$ is naturally equipped with a bialgebra structure making $A$ into a module-algebra. In short:

\begin{thm}
  For any unital algebra $A$, there is a universal bialgebra $B=M(A,A)$ equipped with a map $B\otimes A\to A$ making $A$ into a $B$-module algebra. 
\end{thm}

The discussion further extends to {\it Hopf} algebras: according to \cite[Theorem 14]{prst-lim} the category of Hopf algebras is {\it coreflective} in that of bialgebras, i.e. every bialgebra $B$ admits a universal bialgebra morphism $H\to B$ from a Hopf algebra $H$. Taking for $H$ this envelope of $M(A,A)$, we obtain

\begin{thm}\label{th:initunivh}
  For any unital algebra $A$, there is a universal bialgebra $H=M(A,A)$ equipped with a map $H\otimes A\to A$ making $A$ into an $H$-module algebra.
\end{thm}

Theorem \ref{th:initunivh} is generalized in \cite[Theorem 4.2 and discussion following it]{agv} to what the authors of that paper call {\it $\Omega$-algebras}: vector spaces equipped with maps $A^{\otimes l}\to A^{\otimes r}$ that generalize, say, multiplication maps $A^{\otimes 2}\to A$, units $k\to A$, etc. In a similar direction, we will need to accommodate the following type of situation, of interest in Galois-theoretic considerations. We will have an algebra (or $*$-algebra) $A$, a subalgebra $A'\subseteq A$, and will require the measuring morphisms
\begin{equation*}
  C\otimes A\to A
\end{equation*}
to {\it fix} $A'$ in the sense that
\begin{equation*}
  c\cdot a = \varepsilon(c)a,\ \forall c\in C,\ a\in A'.
\end{equation*}
More generally, one can fix
\begin{itemize}
\item an algebra $A$
\item with a subalgebra $A'\subseteq A$ and
\item an algebra morphism $f:A'\to B$,
\end{itemize}
and ask that the measuring $C\otimes A\to B$ preserve $f$ in the sense that
\begin{equation*}
  c\cdot a = \varepsilon(c) f(a),\ \forall c\in C,\ a\in A'. 
\end{equation*}

To house the various structures we would like our measurings to preserve we introduce the following notion.

\begin{defn}
  In general, a {\it span} in $\cV$ is a diagram of the form 
  \begin{equation*}
    \begin{tikzpicture}[auto,baseline=(current  bounding  box.center)]
      \path[anchor=base] 
      (0,0) node (1) {$\bullet$} 
      +(4,0) node (2) {$\bullet$}
      +(2,-.5) node (d) {$\bullet$}
      ;
      \draw[->] (1) to[bend right=6] node[pos=.5,auto] {$\scriptstyle $} (d);
      \draw[->] (2) to[bend left=6] node[pos=.5,auto] {$\scriptstyle $} (d);
    \end{tikzpicture}
  \end{equation*}
  
  For objects $A$ and $B$ in $\cC$ an {\it $(A,B)$-span} (or span {\it on} $(A,B)$, or $\hom(A,B)$, or just plain span when $A$ and $B$ are understood) is a morphism (in $\cV$) of the form
  \begin{equation}\label{eq:span}
    \begin{tikzpicture}[auto,baseline=(current  bounding  box.center)]
      \path[anchor=base] 
      (0,0) node (1) {$V^{\otimes l}$} 
      +(4,0) node (2) {$V^{\otimes r}$}
      +(2,-.5) node (d) {$\bullet$}
      ;
      \draw[->] (1) to[bend right=6] node[pos=.5,auto] {$\scriptstyle $} (d);
      \draw[->] (2) to[bend left=6] node[pos=.5,auto] {$\scriptstyle $} (d);
    \end{tikzpicture}
  \end{equation}
  for
  \begin{itemize}
  \item an object $V\in \cV$ equipped with a morphism $V\to (A,B)$ (we will frequently simply take $V=(A,B)$, but not always);
  \item non-negative integers $l$ and $r$ (possibly including $0$), where the $0^{th}$ tensor power means the monoidal unit of $\cV$.
  \end{itemize}

  An {\it $(A,B)$-multispan} is a tuple of $(A,B)$-spans.
\end{defn}

Some examples follow, explaining how this fits into the above discussion on measuring.

\begin{exa}\label{ex:mult}
  If $A$ and $B$ are algebras internal to $\cC$, their multiplication operators $m_A:A^{\otimes 2}\to A$ and $m_B$ afford us the following span (where we omit `$\hom$', writing simply $(\bullet,-)$ for $\hom_{\cC}(\bullet,-)$):
  \begin{equation*}
    \begin{tikzpicture}[auto,baseline=(current  bounding  box.center)]
      \path[anchor=base] 
      (0,0) node (lu) {$(A,B)^{\otimes 2}$}
      +(3,0) node (u) {$(A^{\otimes 2},B^{\otimes 2})$}
      +(0,-1) node (ld) {$(A,B)$}
      +(6,-.5) node (r) {$(A^{\otimes 2},B)$}
      ;
      \draw[->] (lu) to[bend left=6] node[pos=.5,auto] {$\scriptstyle $} (u);
      \draw[->] (u) to[bend left=6] node[pos=.5,auto] {$\scriptstyle m_B\circ$} (r);
      \draw[->] (ld) to[bend right=6] node[pos=.5,auto,swap] {$\scriptstyle \circ m_A$} (r);
    \end{tikzpicture}
  \end{equation*}
\end{exa}

\begin{exa}\label{ex:omega}
  More generally, $\Omega$-algebras in the sense of \cite[\S 3.1]{agv} can be treated similarly; any two, say $A$ and $B$, will then give rise to a multispan consisting of one span for each $\omega\in \Omega$.
\end{exa}

\begin{exa}\label{ex:subalg}
  This framework can also accommodate the setup described above, of a morphism $f:A'\to B$ from a subalgebra $A'\subseteq A$ (say $\cV=\cC=\mathrm{Vect}$, for simplicity).  
  \begin{equation*}
    \begin{tikzpicture}[auto,baseline=(current  bounding  box.center)]
      \path[anchor=base] 
      (0,0) node (lu) {$(A,B)$}
      +(0,-1) node (ld) {$(A,B)^{\otimes 0}\cong {\bf 1}$}
      +(4,-.5) node (r) {$(A',B),$}
      ;
      \draw[->] (lu) to[bend left=6] node[pos=.5,auto] {$\scriptstyle \text{restrict}$} (r);
      \draw[->] (ld) to[bend right=6] node[pos=.5,auto,swap] {$\scriptstyle f$} (r);
    \end{tikzpicture}
  \end{equation*}
  where in the bottom arrow ${\bf 1}$ denotes the monoidal unit (the ground field, in the category of vector spaces), and the morphism simply picks out the distinguished morphism $f:A'\to B$.
\end{exa}

\begin{defn}
  Let $C\in \cV$ be a coalgebra in $\cV$, and consider a multispan $\cM$ on $(A,B)$, consisting of maps (\ref{eq:span}). A {\it measuring} by $C$ on $\cM$ is a morphism $\psi:C\to (A,B)$ in $\cV$ that makes all diagrams
  \begin{equation}\label{eq:commspan}
    \begin{tikzpicture}[auto,baseline=(current  bounding  box.center)]
      \path[anchor=base] 
      (0,0) node (1) {$\hom(A,B)^{\otimes l}$}
      +(4,0) node (2) {$\hom(A,B)^{\otimes r}$}
      +(2,-.5) node (d) {$\bullet$}
      +(0,1.5) node (lu) {$C^{\otimes l}$}
      +(4,1.5) node (ru) {$C^{\otimes r}$}
      +(2,2) node (m) {$C$}
      ;
      \draw[->] (1) to[bend right=6] node[pos=.5,auto] {$\scriptstyle $} (d);
      \draw[->] (2) to[bend left=6] node[pos=.5,auto] {$\scriptstyle $} (d);
      \draw[->] (m) to[bend right=6] node[pos=.5,auto] {$\scriptstyle $} (lu);
      \draw[->] (m) to[bend left=6] node[pos=.5,auto] {$\scriptstyle $} (ru);
      \draw[->] (lu) to[bend right=6] node[pos=.5,auto,swap] {$\scriptstyle \psi^{\otimes l}$} (1);
      \draw[->] (ru) to[bend left=6] node[pos=.5,auto] {$\scriptstyle \psi^{\otimes r}$} (2);
    \end{tikzpicture}
  \end{equation}
  (attached to spans in $\cM$) commute.

  We also say that $C$ {\it measures $\cM$}.
\end{defn}

\begin{thm}\label{th:univcoalgv}
  For a $\cV$-enriched and tensored category $\cC$, objects $A,B\in \cC$ and a multispan $\cM$ on $(A,B)$ the category of coalgebras in $\cV$ measuring $\cM$ has a terminal object $M(\cM)$.
\end{thm}
\begin{proof}
  Much as in the proofs of, say, \cite[Theorem 7.0.4]{swe} or \cite[Theorem 3.10]{agv}: first consider the cofree coalgebra $C(A,B)$ on the $\cV$-object $\hom(A,B)$, and then take the colimit of the category of subobjects $C\to C(A,B)$ consisting of objects for which the diagrams (\ref{eq:commspan}) commute.
\end{proof}

\begin{cor}\label{cor:llst}
  For any
  \begin{itemize}
  \item algebra $A$ over a field $k$;
  \item $k$-algebra $A$ equipped with a functional $\tau:A\to k$;
  \item complex $*$-algebra or
  \item complex $*$-algebra equipped with a functional $\tau:A\to \bC$. 
  \end{itemize}
  there is a universal coalgebra ($*$-coalgebra in the two latter cases) measuring the entirety of the structure. Furthermore, the same holds if we enrich this structure with a subalgebra $A'\subseteq A$ and require that the measuring fix it.
\end{cor}
\begin{proof}
  These are all particular instances of Theorem \ref{th:univcoalgv}, for various choices of category and/or multispan. To take the case of an embedding of $*$-algebras $A'\subseteq A$, for example, one takes
  \begin{equation*}
    \cV = \cC = \mathrm{Vect}^*_{\bC}
  \end{equation*}
  and spans corresponding to
  \begin{itemize}
  \item the multiplication and unit of $A$, as in Examples \ref{ex:mult} and \ref{ex:omega}, as well as
  \item the embedding $A'\subseteq A$, by specializing Example \ref{ex:subalg} to $B=A$ and $f=\id_{A'}$.
  \end{itemize}
  One can further handle functionals as instances of Example \ref{ex:omega} again.
\end{proof}

Next, we consider the issue of universal measuring {\it bi}algebras and Hopf algebras. To that end, we will have to take into account the interaction between several multispans; in turn, this requires some terminology.

\begin{defn}
  The {\it shape} of a span (\ref{eq:span}) is the pair $(l,r)$. Similarly, the shape of a multispan is the tuple of shapes of its individual spans (the ordering of that tuple is part of the data).
\end{defn}

Now consider two spans: (\ref{eq:span}), denoted $S(A,B)$ and an analogous one, $S(B,C)$, of the same shape $(l,r)$. Their {\it tensor product} $S(B,C)\otimes S(A,B)$ is simply their tensor product in the category of spans:
\begin{equation}\label{eq:tensspan}
  \begin{tikzpicture}[auto,baseline=(current  bounding  box.center)]
    \path[anchor=base] 
    (0,0) node (1) {$(V_{B,C}\otimes V_{A,B})^{\otimes l}$} 
    +(4,0) node (2) {$(V_{B,C}\otimes V_{A,B})^{\otimes r}$}
    +(2,-.5) node (d) {$\bullet\otimes\bullet$}
    ;
    \draw[->] (1) to[bend right=6] node[pos=.5,auto] {$\scriptstyle $} (d);
    \draw[->] (2) to[bend left=6] node[pos=.5,auto] {$\scriptstyle $} (d);
  \end{tikzpicture}
\end{equation}
where
\begin{itemize}
\item $V_{A,B}\in \cV$ is the object mapping to $(A,B)$ that is part of the data specifying $S(A,B)$, and similarly for $V_{B,C}$;
\item we have made free use of the symmetry of $\cV$ to permute the top tensorands;
\item the two unidentified bullets at the bottom are the two respective tips of the original spans $S(A,B)$ and $S(B,C)$. 
\end{itemize}

\begin{defn}
  Let $\cM_1$ and $\cM_2$ be two multispans of the same shape, consisting respectively of, say, spans $S_{1,j}$ and $S_{2,j}$ of respective shapes $(l_j,r_j)$. The {\it tensor product} $\cM_1\otimes \cM_2$ is the tuple of tensor products $S_{1,j}\otimes S_{2,j}$ in the sense of (\ref{eq:tensspan}). 
\end{defn}

An example will illustrate the point. 

\begin{exa}\label{ex:comp}
  Consider two spans as in Example \ref{ex:subalg}, one, $S(f)$, for a map $f:A'\to B$, $A'\subseteq A$ and similarly, $S(g)$ for $g:B'\to C$, $B'\subseteq B$. Assume furthermore that $f$ and $g$ are composable, in the sense that the image of $f$ is contained in $B'$. We then have a morphism of spans from
  \begin{equation*}
    S(g)\otimes S(f)=
    \begin{tikzpicture}[auto,baseline=(current  bounding  box.center)]
      \path[anchor=base] 
      (0,0) node (1) {${\bf 1}$} 
      +(8,0) node (2) {$(B,C)\otimes (A,B)_{A',B'}$}
      +(4,-.5) node (d) {$(B',C)\otimes (A',B')$}
      ;
      \draw[->] (1) to[bend right=6] node[pos=.5,auto] {$\scriptstyle $} (d);
      \draw[->] (2) to[bend left=6] node[pos=.5,auto] {$\scriptstyle $} (d);
    \end{tikzpicture}
  \end{equation*}
  where
  \begin{equation*}
    (A,B)_{A',B'}\subseteq (A,B) \in \cV
  \end{equation*}
  is the subobject consisting of morphisms that map $A'$ to $B'$, i.e. the pullback
  \begin{equation*}
    \begin{tikzpicture}[auto,baseline=(current  bounding  box.center)]
      \path[anchor=base] 
      (0,0) node (1) {$(A',B')$} 
      +(4,0) node (2) {$(A,B),$}
      +(2,-.5) node (d) {$(A',B)$}
      +(2,.5) node (u) {$(A,B)_{A',B'}$}
      ;
      \draw[->] (1) to[bend right=6] node[pos=.5,auto,swap] {$ $} (d);
      \draw[->] (2) to[bend left=6] node[pos=.5,auto] {$ $} (d);
      \draw[->] (u) to[bend right=6] node[pos=.5,auto] {$$} (1);
      \draw[->] (u) to[bend left=6] node[pos=.5,auto] {$$} (2);
    \end{tikzpicture}
  \end{equation*}  
  to
  \begin{equation*}
    S:=
    \begin{tikzpicture}[auto,baseline=(current  bounding  box.center)]
      \path[anchor=base] 
      (0,0) node (1) {${\bf 1}$} 
      +(6,0) node (2) {$(A,C)$}
      +(3,-.5) node (d) {$(A',C)$}
      ;
      \draw[->] (1) to[bend right=6] node[pos=.5,auto] {$\scriptstyle $} (d);
      \draw[->] (2) to[bend left=6] node[pos=.5,auto] {$\scriptstyle $} (d);
    \end{tikzpicture}
  \end{equation*}
  defined as follows:
  \begin{itemize}
  \item ${\bf 1}\to {\bf 1}$ is the identity;
  \item the map $(B,C)\otimes(A,B)_{A',B'}\to (A,C)$ is just composition,
  \item as is $(B',C)\otimes (A',B')\to (A',C)$.
  \end{itemize}
\end{exa}

\begin{prop}\label{pr:multialg}
  Under the hypotheses of Theorem \ref{th:univcoalgv}, let $A$, $B$ and $C$ be three objects in $\cC$ and $\cM_{A,B}$ and $\cM_{B,C}$ multispans of the same shape (between the objects indicated).

  A morphism
  \begin{equation*}
    \cM_{B,C}\otimes \cM_{A,B}\to \cM_{A,C}
  \end{equation*}
  of multispans uniquely induces a corresponding morphism
  \begin{equation*}
    M(\cM_{B,C})\otimes M(\cM_{A,B})\to M(\cM_{A,C})
  \end{equation*}
  of coalgebras.
\end{prop}
\begin{proof}
  Immediate from the universality property of the measuring coalgebra $M(\cM)$ of a multispan $\cM$ as a terminal object.
\end{proof}

\begin{cor}\label{cor:isbialg}
  In any of the cases listed in Corollary \ref{cor:llst} the universal measuring coalgebra is automatically a bialgebra.
\end{cor}
\begin{proof}
  This are all specializations of Proposition \ref{pr:multialg}, again choosing specific $\cC$, $\cV$ and multispans; the associativity of the bialgebra structure follows from the uniqueness in Proposition \ref{pr:multialg}.

  To fix ideas we focus, say, on the universal $*$-coalgebra measuring $A$ and fixing a specific $*$-subalgebra $A'\subseteq A$. We can then specialize Example \ref{ex:comp} to the case $B=C=A$ and
  \begin{equation*}
    f:A'\to A = \text{ the embedding }A'\subseteq A,
  \end{equation*}
  (i.e. the restriction of $\mathrm{id}_A$ to $A'$) to conclude.
\end{proof}

\begin{thm}\label{th:univhpf}
  In any of the cases listed in Corollary \ref{cor:llst} there is a universal measuring Hopf algebra.
\end{thm}
\begin{proof}
  Consider the measuring bialgebra $B(\cM)$ attached to the multispan $\cM$ per Corollary \ref{cor:isbialg}, and then construct the corresponding Hopf algebra mapping universally to it:
  \begin{equation*}
    H(\cM)\to B(\cM). 
  \end{equation*}
  As noted while setting up the present framework, \cite[Theorem 54 1.]{prst-ft1} ensures that such a Hopf algebra exists, finishing the proof.
\end{proof}

%%%%%%%%%%%%%%%%%%%%%%%%%%%%%%%%%%%%%%%%%%%%%%%%%%%%%%%%%%%%%%%%%%%%%%%%%%%%%
%%%%%%%%%%%%%%%%%%%%%%%%%%%%%%%%%%%%%%%%%%%%%%%%%%%%%%%%%%%%%%%%%%%%%%%%%%%%%
\section{Quantum Galois group - definition and existence}\label{section-def-exist}

In this section, we define what we call the quantum Galois group of an inclusion of $\2$-factors. Examples are given in the next section. 

%%%%%%%%%%%%%%%%%%%%%%%%%%%%%%%%%%%%%%%%%%%%%%%%%%%%%%%%%%%%%%%%%%%%%%%%%%%%%
\subsection{Generalities on subfactors} 
For the convenience of the reader, we recall some of the terminology and results from subfactor theory. Some references are \cites{ekquantsymm,ghj,pimsnerpopa,takesaki3,jonesindex}.

\medskip\noindent

Let $M$ be a finite von Neumann algebra with a fixed normal faithful trace $\t$, $\t(1)=1$. The Hilbert norm on $M$ given by $\t$ is denoted as $\|x\|_2=\t(x^*x)^{1/2}$ and let $L^2(M,\t)$ be the completion of $M$ in this norm. Thus $L^2(M,\t)$ is the Hilbert space of the GNS representation of $M$, given by $\t$, and $M$ acts on $L^2(M,\t)$ by left multiplication. This representation is called the standard representation. The canonical conjugation on $L^2(M,\t)$ is denoted by $J$. It acts on the dense subspace $M \subset L^2(M,\t)$ by $Jx=x^*$. Then $J$ satisfies $JMJ=M'$ and in fact $JxJ$ is the operator of multiplication on the right with $x^*$: $JxJ(y)=yx^*$, $y \in M \subset L^2(M,\t)$.

\medskip\noindent

If $N \subset M$ is a von Neumann subalgebra with the same unit, then $E_N$ denotes the unique $\t$-preserving conditional expectation of $M$ onto $N$. $E_N$ is in fact the restriction to $M$ of the orthogonal projection of $L^2(M,\t)$ onto $L^2(N,\t)$, $L^2(N,\t)$ being the closure of $N$ in $L^2(M,\t)$. This orthogonal projection is denoted by $e_N$. We list some properties of $e_N$.

\medskip\noindent

\begin{enumerate}
    \item $e_Nxe_N=E_N(x)$, for $x \in M$.
    \item If $x \in M$ then $x \in N$ if and only if $e_Nx=xe_N$.
    \item $N'=(M \cup \{e_N\})''$.
    \item $J$ commutes with $e_N$.
\end{enumerate}

\medskip\noindent

Let $M_1$ denote the von Neumann algebra on $L^2(M,\t)$ generated by $M$ and $e_N$. It follows that $M_1=JN'J$. $M_1$ is called the basic construction for $N \subset M$. We record two properties of $M_1$.

\medskip\noindent

\begin{enumerate}
    \item[(5)] $M_1$ is a factor if and only if $N$ is. 
    \item[(6)] $M_1$ is finite if and only if $N'$ is.
\end{enumerate}

\medskip\noindent

If $M$ is a finite factor and acts on the Hilbert space $H$, the Murray-von Neumann coupling constant $\mathrm{dim}_M H$ is defined as $\t([M'\xi])/\t'([M\xi])$, where $0\neq \xi \in H$ and for $A\subset B(H)$ a von Neumann algebra $[A\xi]$ is the orthogonal projection onto $\ol{A\xi}$. The definition is independent of $\xi \neq 0$. For a pair of finite factors $N \subset M$, the Jones index of $N$ inside $M$, denoted $[M:N]$, is defined to be the number $\mathrm{dim}_N L^2(M,\t)$. Some properties of the Jones index are listed below.

\medskip\noindent

\begin{enumerate}
    \item[(7)] $[M:M]=1$. If $N\subset P \subset M$ then $[M:P][P:N]=[M:N]$.
    \item[(8)] If $[M:N] < \infty$ then $N' \cap M$ is finite dimensional.
    \item[(9)] If $[M:N] < \infty$ then $M_1$ is a finite factor and the canonical trace on $M_1$, say $\t_1$, has the Markov property: \[\t_1(e_Nx)=\frac{1}{[M:N]}\t(x) \quad \forall x \in M.\]
    \item[(10)] $[M_1:M]=[M:N]$.
\end{enumerate}

\medskip\noindent

We now recall a theorem due to Pimsner and Popa.

\begin{thm}
\label{pp basis}
% Let $N \subset M$ be type $\2$ von Neumann algebras with finite dimensional centres and let $\t_M$ be a faithful normal trace on $M$ for which $N'$ is finite on $L^2(M,\t_M)$. 
Let $N \subset M$ be a pair of finite factors with $[M:N] < \infty$.Then
\begin{enumerate}
\item As a right module over $N$, the algebra $M$ is projective of finite type.
\item $M_1=\{\sum_{j=1}^na_je_Nb_j \mid n \geq 1, a_j,b_j \in M\}$.
\item If $\a: M \to M$ is a right $N$-module map, then $\a$ extends uniquely to an element of $M_1$ on $L^2(M,\t)$.
\item If $x \in M_1$ then $x(M)\subset M$, where $M$ is viewed as a dense subspace of $L^2(M,\t)$.
\end{enumerate}
\end{thm}

% In the course of the proof, they exhibit a Pimsner-Popa basis which immediately. Let us recall that. A Pimsner-Popa basis consists of elements $\{v_i \mid 1\leq i \leq n$\} of $M$ with the following properties:
% \begin{enumerate}[a)]
% \item $E_N(v_i^*v_j)=0$ if $i\neq j$.
% \item $f_i:=E_N(v_i^*v_i)$ is a projection in $N$, $v_if_i=v_i$, and $E_N(v_i^*x)=f_iE_N(v_i^*x)$, for $1\leq i \leq n$ and $x \in M$.
% \item Every $x$ in $M$ has a unique expansion \[x=\sum_{i=1}^nv_iy_i, \quad y_i \in N.\] In fact $v_iy_i=v_iE_N(v_i^*x)$.
% \end{enumerate} 

% Now for the proof of iii), we observe that if $\a : M \to M$ is right $N$-linear then for $x \in M$, \[\a(x)=\sum_i\a(v_i)E_N(v_i^*x)=\sum_i\l(\a(v_i))\circ E_N \circ \l(v_i^*)(x),\] where $\l(y)$ denotes left multiplication by $y$. This implies that a right $N$-linear map such as $\a$ is $\|\cdot\|_2$-continuous, being a composition of left multiplication operators and hence extends uniquely to $L^2(M,\t_M)$. The unique element of $\<M,e_N\>$ corresponding to $\a$ is then given by $\sum_i\a(v_i)e_Nv_i^*$.

As a corollary, one has

\begin{cor}\label{end iso}
% Let $N \subset M$ be a pair of von Neumann algebras of type $\2$ having finite dimensional centres and suppose that $N$ is of finite index in $M$. Let $\t_M$ be a faithful normal trace on $M$ with $e_N$ and $E_N$ defined via $\t_M$. 
Let $N \subset M$ be a pair of finite factors with $[M:N] < \infty$. Then
\begin{enumerate}
    \item $\en(M_N) \cong M_1$ as $\bC$-algebras, and 
    \item $M \tens_N M \isom M_1$ as $(N,N)$-bimodules.
\end{enumerate} 
\end{cor} 

We have the following proposition.

\begin{prop}\label{fd}
Let $N \subset M$ be a pair of finite index $\2$ factors. Then $\en({}_N M_N)$ is finite dimensional.
\end{prop}

\begin{proof}
By the above Corollary, $\en({}_N M_N)\cong N' \cap M_1$. The result follows from (7), (8) and (10) above.
\end{proof}

%%%%%%%%%%%%%%%%%%%%%%%%%%%%%%%%%%%%%%%%%%%%%%%%%%%%%%%%%%%%%%%%%%%%%%%%%%%%%
\subsection{The existence theorem} 
We now define the quantum Galois group of a pair of finite factors $N \subset M$.

\begin{defn}\label{def:qgal}
Let $N \subset M$ be a pair of finite factors. Let $\mathrm{C}(N\subset M)$ be the category whose 
\begin{itemize}
    \item objects are Hopf $\*$-algebras $Q$ admitting an action on $M$ making it a module $\*$-algebra such that $N \subset M^Q$, where $M^Q$ is the invariant subalgebra;
    \item morphisms between two objects, say $Q$ and $Q'$, are Hopf $\*$-algebra morphisms $\phi : Q \to Q'$ such that the following diagram commutes:
    \begin{equation}
        \begin{tikzcd}
            Q \tens M \arrow[r,"\phi \tens \id"]\arrow[dr] & Q' \tens M \arrow[d]\\
            & M
        \end{tikzcd}
    \end{equation}
\end{itemize} where the unadorned arrows are the respective actions. We define the quantum Galois group of the inclusion $N \subset M$ denoted $\QGal(N \subset M)$ to be a terminal object of the category $\mathrm{C}(N \subset M)$.
\end{defn}

By definition, $\QGal(N \subset M)$ is unique up to unique isomorphism. We also introduce the following. %There is no reason for it to exists, however, as we shall prove below, under some assumption, it always does. It is useful, in fact necessary, to introduce the following.

\begin{defn}
  Let $\mathrm{C}_{\t}(N\subset M)$ be the full subcategory of $\mathrm{C}(N\subset M)$ consisting of Hopf $\*$-algebras admitting a $\t$-preserving action on $M$. A terminal object in this category is denoted as $\QGal_{\t}(N\subset M)$.
\end{defn}

Existence is immediate as a consequence of Theorem \ref{th:univhpf} above:

\begin{thm}\label{th:ex}
  For any inclusion $N\subset M$ of finite factors the quantum Galois groups $\QGal(N\subset M)$ and $\QGal_{\t}(N\subset M)$ both exist.
\end{thm}

\begin{cor}\label{min}
  Let $N \subset M$ be an irreducible pair of finite factors with $[M:N] < \infty$. Then the action of $\QGal_{\t}(N \subset M)$ on $M$ is minimal. Furthermore, the invariant subalgebra $M^{\QGal_{\t}(N \subset M)}$ is a factor with $[M:M^{\QGal_{\t}(N \subset M)}] < \infty$.
\end{cor}

\begin{proof}
  Denote by $P$ the invariant subalgebra $M^{\QGal_{\t}(N \subset M)}$. Thus $N \subset P \subset M$ and therefore $P' \cap P \subset P' \cap M \subset N' \cap M=\bC 1_M$, whence the result follows.
\end{proof}

With Theorem \ref{th:ex} in hand, the rest of the paper is devoted to identifying these universal objects in various particular cases of interest.

\begin{prop}
  Let $N\subset M$ be a pair of finite factors with $[M:N] < \infty$ and let $H$ be a Hopf $\*$-algebra belonging to the category $\mathrm{C}(N\subset M)$. If the unique normalized trace $\t$ on $M$ is preserved under the $H$-action, i.e., $\t(h\cdot x)=\ve(h)\t(x)$, for all $x \in M$, then so is $\t_1$. Here, the $H$-action on $M_1$ is obtained from Corollary \ref{end iso}.
\end{prop}

\begin{proof}
  First observe that \[\t_1(\sum_ja_je_Nb_j)=\sum_j\t_1(e_Nb_ja_j)=\sum_j\frac{1}{[M:N]}\t(b_ja_j).\] The first equality follows from the traciality of $\t_1$. The second from the Markov property (9) above. So it suffices to check the claim for elements of the form $e_Nx$, where $x \in M$. Now,
\[
\begin{aligned}
\t_1(h \cdot (e_Nx))&=\t_1(e_N(h \cdot x))\\
&=\frac{1}{[M:N]}\t(h\cdot x)\\
&=\frac{1}{[M:N]}\ve(h)\t(x)\\
&=\ve(h)\t_1(e_Nx).
\end{aligned}
\] 
The first equality follows from the fact that $H$ acts trivially on $e_N$ which can be seen as follows. Each $h \in H$, viewed as a right $N$-linear endomorphism of $M$ extends uniquely to a bounded linear map on $L^2(M,\t)$ by (3) of Theorem \ref{pp basis}. Moreover, $L^2(N,\t)$ is an invariant subspace under such a map. The $H$-action is $\*$-preserving, hence the invariant subspace becomes reducing.
\end{proof}

%Summarizing,

% \begin{thm}
% Let $H$ be a Hopf algebra and $N\subset M$ is a pair of finite index $\2$ factors such that $M$ is an $H$-module algebra and $N \subset M^H$, where $M^H$ is the invariant subalgebra. Moreover, suppose that $H$ preserves $\t_M$, where $\t_M$ is the unique normal trace. Then there exists a pairing $\<-,-\> : H \tens Q_{aut}(\en({}_N M_N),\t_1) \to \bC$ such that the $H$-action factors through the dual action of $Q_{aut}(\en({}_N M_N),\t_1)$. Here, $Q_{aut}$ is in Wang's \cite{wangqsymm} sense.
% \end{thm} 

Next, we restate Wang's \cite{wangqsymm} result on quantum symmetry groups of finite spaces in the action picture.
% which will be used below for the existence theorem. However, the statements of the theorems are interesting in their own right.

\begin{thm}\label{th:qfinsp}
  Let $(B,\tau)$ be a finite-dimensional $*$-algebra equipped with a state, and $Q_{aut}(B,\tau)$ its quantum automorphism group, as in \cite[Theorem 6.1]{wangqsymm}. Then $\QGal_{\t}(\bC\subset B)$ is isomorphic to the finite dual Hopf $*$-algebra $Q_{aut}(B,\tau)^{\circ}$.
\end{thm}
\begin{proof}
  Let $H$ be a Hopf algebra acting in a $\tau$-preserving fashion. This translates to a $\tau$-preserving {\it co}action by $H^{\circ}$, i.e. a morphism
  \begin{equation*}
    Q_{aut}(B,\tau)\to H^{\circ}.
  \end{equation*}
  But in turn, because the finite dual is a self-adjoint contravariant functor on the category of Hopf $*$-algebras, this gives (naturally in $H$) a Hopf $*$-algebra morphism
  \begin{equation*}
    H\to Q_{aut}(B,\tau)^{\circ}. 
  \end{equation*}
  The naturality in $H$ means precisely that $Q_{aut}(B,\tau)^{\circ}$ has the requisite universality property, finishing the proof. 
\end{proof}

\begin{rem}
  The proof above has cognates in the literature; see e.g. \cite[Theorem 3.20]{agv}.
\end{rem}

Theorem \ref{th:qfinsp} together with Proposition \ref{fd} imply

\begin{prop}\label{existence}
  Let $N\subset M$ be a pair of finite factors with $[M:N] < \infty$ and let $H$ be a Hopf $\*$-algebra belonging to the category $\mathrm{C}(N\subset M)$. Moreover, assume that the unique normalized trace $\t$ on $M$ is preserved under the $H$-action.

  The induced $H$-action on $\en({}_N M_N)$ factors through the dual action of a Hopf $\*$-subalgebra of the dual $Q^\*_{aut}(\en({}_N M_N),\t_1)$ of $Q_{aut}(\en({}_N M_N), \t_1)$.
  % \item there exists a universal Hopf $\ast$-algebra, to be denoted by $\QGal(N\subset M)$, which has a $\*$-compatible action on $M$ such that $N $ is in the invariant subalgebra $M^H$;
  % \item this universal Hopf $\*$-algebra consists of those elements $h \in Q^*_{aut}(\en({}_N M_N),\t_1)$ such that \[h \cdot (xy)=(h_{(1)}\cdot x)(h_{(2)}\cdot y)\] for all $x,y \in M$.
  % \end{enumerate}
\end{prop}

From the above Proposition, it is possible to extract another proof of Theorem \ref{th:ex} as follows. One applies Zorn's lemma to the poset of Hopf $*$-subalgebras of $Q^\*_{aut}(\en({}_N M_N),\t_1)$ that acts on $M$ (or equivalently, Hopf $*$-algebras acting on $\mathrm{End}({}_NM_N)$ via an action on $M$) which is nonempty (since it contains the trivial one dimensional Hopf algebra) and directed by inclusion. Hence we obtain a maximal element. This maximal element acts on $M$ and by the Proposition above, is the quantum Galois group $\QGal_{\t}(N\subset M)$.

The entire discussion can be phrased more succinctly, using the notion of Hopf action on an object in a monoidal category \cite[\S 3]{hopftensor}. We are grateful to the referee for pointing this out and suggesting the connection. 

\begin{thm}
  Let $N \subset M$ be a pair of finite factors with $[M:N] < \infty$.

  $\QGal(N\subset M)$ ($\QGal_{\t}(N\subset M)$) is isomorphic to the ($\tau$-preserving) quantum automorphism group of the algebra object $M$, in the monoidal category $\mathrm{Bimod}_{N-N}$.
\end{thm}
\begin{proof}
  This is virtually tautological, given that $N$-fixing Hopf actions on $M$ can be phrased in terms of $N$-bimodule morphisms:
  \begin{itemize}
  \item Fixing $N$ pointwise means compatibility with the unit $N\to M$ the algebra object
    \begin{equation*}
      M\in \mathrm{Bimod}_{N-N}
    \end{equation*}
    in the sense of the bottom right-hand diagram in \cite[paragraph preceding Remark 4, p.75]{hopftensor}.
  \item Preserving the multiplication
    \begin{equation*}
      M\otimes_NM\to M
    \end{equation*}
    is expressible as the commutativity of the bottom left-hand diagram in the same display.
  \item Finally, in the $\tau$-preserving case, the compatibility of the action with $\tau$ can be rephrased as leaving invariant the $N$-bimodule map $M\to N$ induced by the canonical tracial states (i.e. the expectation of $M$ onto $N$).
  \end{itemize}
  Having thus translated $\mathrm{C}(N\subset M)$-structures into actions on algebra objects in $\mathrm{Bimod}_{N-N}$, the conclusion follows. 
\end{proof}

\section{Examples of quantum Galois groups}\label{section-examples}
In this section, we compute the quantum Galois group of some of the generic examples of finite factor inclusions.

%%%%%%%%%%%%%%%%%%%%%%%%%%%%%%%%%%%%%%%%%%%%%%%%%%%%%%%%%%%%%%%%%%%%%%%%%%%%%
\subsection{Inclusions arising from crossed products}
\label{subseccrossed}
One of the standard ways of producing a finite-index pair of finite factors from a given one, say $M$ is by taking crossed product with an outer action of a finite dimensional Hopf $C^*$-algebra $H$. It is also well-known that such a pair $M \subset M\rtimes H$ is a generic example of an irreducible ``depth-2'' inclusion. More precisely, let $N\subset M$ be a pair of finite factors with $[M:N] < \infty$ and $N'\cap M=\bC 1_M$. Then, by \cite{szymcrossed}, there exists a finite dimensional Hopf $C^*$-algebra $H$ which acts on $N$ such that $M=N\rtimes H$. We compute the quantum Galois group of such an inclusion in this subsection. We work in a bit more general situation as it needs no extra effort.

\medskip\noindent 

Let $H$ be a (not necessarily finite dimensional) Hopf algebra and $A$ be an $H$-module algebra. Denote by $A \rtimes H$ the smash product of $A$ by $H$. We recall the following lemma.

\begin{lem}
Let $V \in \Hom_{\bC}(H, A \rtimes H)$ be the map
\begin{equation}
    V(h)=1\rtimes h.
\end{equation} 
Then $V$ is convolution invertible and ``innerifies" the $H$-action, i.e.,
\begin{equation}
    h\cdot x \rtimes 1=V(h_1)(x\rtimes 1)V^{-1}(h_2),
\end{equation}
where $h \in H$, $x \in A$, $\D h=h_1 \tens h_2.$
\end{lem}

\begin{proof}
This is standard. But we nevertheless check the details. Let $V^{-1} : H \to A\rtimes H$ be defined as $V^{-1}(h)=1\rtimes S(h)$. Then we compute
\[
\begin{aligned}
VV^{-1}(h)&=V(h_1)V^{-1}(h_2)\\
&=(1 \rtimes h_1)(1 \rtimes S(h_2))\\
&=h_1\cdot 1\rtimes h_2S(h_3)\\
&=\ve(h_1)\rtimes \ve(h_2)\\
&=1\rtimes \ve(h)1.
\end{aligned}
\] The other equality can be proven similarly. Now, for $x \in A$,

\[
\begin{aligned}
V(h_1)(x\rtimes 1)V^{-1}(h_2)&=(1\rtimes h_1)(x\rtimes 1)(1\rtimes S(h_2))\\
&=(h_1\cdot x \rtimes h_2)(1\rtimes S(h_3))\\
&=(h_1 \cdot x)(h_2 \cdot 1)\rtimes h_3S(h_4)\\
&=h_1 \cdot x \rtimes h_2S(h_3)\\
&=h_1 \cdot x\rtimes \ve(h_2)\\
&=h\cdot x\rtimes 1.
\end{aligned}
\]
\end{proof}

Let $Q$ be a Hopf algebra such that $A \rtimes H$ is $Q$-module algebra and $A \subset (A \rtimes H)^{Q}$, where $(A \rtimes H)^{Q}$ is the invariant subalgebra. Such a Hopf algebra exists; for example, let $\widehat{H}$ be a Hopf algebra in duality with $H$ via
\begin{equation*}
  \<-,-\> : \widehat{H} \tens H \to \bC
\end{equation*}
(see the discussion at the beginning of Section \ref{se.prel}). For $u \in \widehat{H}$, $x \in A$ and $h \in H$, define
\begin{equation}
    u \cdot (x \rtimes h)=x \rtimes (u \harp h),
\end{equation}
where $u \harp h=h_1\<u,h_2\>$. Then it is clear that the $\widehat{H}$-action is one such example. What we show below is that this example is the universal example, under certain conditions. By universality we mean that there should exist a Hopf algebra morphism $\phi: Q \to \widehat{H}$ such that the following diagram commutes:
\begin{equation}\label{eq:qh}
  \begin{tikzcd}
    Q \tens (A\rtimes H) \arrow[dr] \arrow[r, "\phi \tens 1"] & \widehat{H} \tens (A \rtimes H) \arrow[d] \\
    & A \rtimes H
  \end{tikzcd}        
\end{equation} 
\noindent
Observe that, a necessary condition for this to happen is that for $q \in Q$, $h \in H$,
\begin{equation}
    q \cdot (1\rtimes h)=\phi(q)\cdot (1\rtimes h)=1\rtimes h_1\<\phi(q),h_2\>.
\end{equation} 
That is $Q$ takes $H$ into $H$ in a very special way. We first achieve this.

\begin{prop} \label{cent}
Let $q \in Q$, thought of as a map from $H \to A\rtimes H$, $h \mto q \cdot(1\rtimes h)$. Then for each $h \in H$,
\begin{equation}
    V^{-1}q(h) \in A' \cap (A\rtimes H),
\end{equation}
where $V^{-1}q$ is the convolution product, $A' \cap (A \rtimes H)$ is the commutant of $A$ in $A \rtimes H$.
\end{prop}

\begin{proof}
Let $x \in A$ and $h \in H$. We compute
\[
\begin{aligned}
(x \rtimes 1)V^{-1}(h_1)q(h_2)&=V^{-1}(h_1)V(h_2)(x \rtimes 1)V^{-1}(h_3)q(h_4)\\
&=V^{-1}(h_1)(h_2\cdot x \rtimes 1)q(h_3)\\
&=V^{-1}(h_1)q \cdot ((h_2\cdot x \rtimes 1)(1 \rtimes h_3))\\
&=V^{-1}(h_1)q \cdot ((1\rtimes h_2)(x \rtimes 1))\\
&=V^{-1}(h_1)q(h_2)(x \rtimes 1).
\end{aligned}
\] The second equality follows from Lemma 1.1, the third and the fifth follow from the fact that $Q$ acts trivially on $A$. Therefore, we are done.
\end{proof}

This leads to

\begin{prop}\label{pr:clsf}
  Let $H$ be a Hopf algebra, $A$ an $H$-module algebra, and $A'$ the commutant of $A$ in $A\rtimes H$. We then have a linear space isomorphism
  \begin{equation*}
    \Hom(H,A') \cong \en({}_A A\rtimes H_A),
  \end{equation*}
  sending $\psi:H\to A'$ the unique $A$-bimodule endomorphism of $A\rtimes H$ sending $h\in H\subset A\rtimes H$ to $h_1\psi(h_2)$. 
\end{prop}
\begin{proof}
  Proposition \ref{cent} shows that every $A$-bimodule endomorphism $q$ of $A\rtimes H$ is uniquely of the desired form, by recovering $\psi:H\to A'$ as $V^{-1} q|_H$. On the other hand, the fact that for {\it every} $\psi:H\to A'$ the map
  \begin{equation*}
    a\rtimes h\mapsto (a\rtimes h_1)\psi(h_2)
  \end{equation*}
  is an $A$-bimodule endomorphism of $A\rtimes H$ is an immediate check.
\end{proof}

In particular, for outer actions we obtain

\begin{cor}\label{cor:irr}
  Let $A$ be a module-algebra over a Hopf algebra $H$, and assume the action is outer. Then, we have an algebra isomorphism
  \begin{equation*}
    H^*\cong \en({}_A A\rtimes H_A)
  \end{equation*}
  given by sending $\lambda\in H^*$ to
  \begin{equation*}
    A\rtimes H\ni a\rtimes h\mapsto a\rtimes h_1\lambda(h_2). 
  \end{equation*}
\end{cor}

% % Before proving the Corollary, we remark that the condition $A' \cap (A \rtimes H)=\bC$ is also expressed by saying that the action is outer.

% % \begin{proof}[Proof of Corollary \ref{irr}]
% % By the previous Proposition, for each $q \in Q$ and $h \in H$ there exists $\l_q(h) \in \bC$ such that $V^{-1}q(h)=\l_q(h)(1\rtimes 1)$. Let $\L_q \in \Hom_{\bC}(H,A\rtimes H)$ be defined as
% % \begin{equation}
% %     \L_q(h)=1\rtimes \l_q(h)1.
% % \end{equation}
% % Then $V^{-1}q=\L_q$ which implies $q=V\L_q$. So for each $h \in H$, \[q\cdot (1\rtimes h)=V(h_1)\L_q(h_2)=(1\rtimes h_1)(1\rtimes \l_q(h_2)1)=1\rtimes h_1\l_q(h_2),\] which was to be obtained. Uniqueness follows from applying $\ve$.
% % \end{proof}
% %

Going back to the setting of (\ref{eq:qh}), Corollary \ref{cor:irr} says that for each $q \in Q$ and $h \in H$ there exists $\l_q(h) \in \bC$ such that $V^{-1}q(h)=\l_q(h)(1\rtimes 1)$. We can now define a pairing between $Q$ and $H$, from which universality follows automatically: set
\begin{equation}
    \<-,-\>: Q \tens H \to \bC, \quad \<q,h\>=\l_q(h)=(1 \rtimes \ve)(q \cdot (1\rtimes h)).
\end{equation}
We show that this defines a dual pairing.

\medskip

{\bf Step 1}: $\<qq',h\>=\<q \tens q, \D h\>=\<q,h_1\>\<q',h_2\>$ holds: For, by associativity, \[qq' \cdot (1\rtimes h)=q\cdot (1\rtimes h_1\l_{q'}(h_2))=1\rtimes h_1\l_q(h_2)\l_{q'}(h_3).\] Therefore \[\<qq',h\>=\ve(h_1)\l_q(h_2)\l_{q'}(h_3)=\l_q(h_1)\l_{q'}(h_2)=\<q,h_1\>\<q',h_2\>.\]

\medskip

{\bf Step 2}: $\<q,hh'\>=\<q_1,h\>\<q_2,h'\>$ holds: Since $A \rtimes H$ is a $Q$-module algebra, we have \[q \cdot (1\rtimes hh')=q_1 \cdot (1\rtimes h)q_2 \cdot (1\rtimes h').\] Now \[q\cdot (1\rtimes hh')=1\rtimes h_1h'_1\l_q(h_2h'_2)\] and \[q_1 \cdot (1\rtimes h)q_2 \cdot (1\rtimes h')=h_1\l_{q_1}(h_2)h'_1\l_{q_2}(h'_2)=h_1h'_1\l_{q_1}(h_2)\l_{q_2}(h'_2).\] Applying $\ve$ yields the result.

\medskip

{\bf Step 3}: $\<1,h\>=\ve(h)$ obviously holds.

\medskip

{\bf Step 4}: $\<q,1\>=\ve(q)$ too is obvious.

\medskip

{\bf Step 5}: $\<q,S(h)\>=\<S(q),h\>$ holds: Obviously, the pairing defines a bialgebra morphism from $Q \to \widehat{H}$ and since a bialgebra morphism is in fact a Hopf algebra morphism, we get the result. Nevertheless, we present an argument which uses only the pairing and which is itself pretty! For that, we first recall that the map
\begin{equation}
    H \tens H \to H \tens H, \quad x \tens y \mto x_1 \tens x_2y,
\end{equation}
is a bijection. Now, using this, for $h \in H$, find $x^i, y_i \in H$ such that $\sum_ix^i_1 \tens x^i_2y_i=h \tens 1$. We compute,
\allowdisplaybreaks{
\[
\begin{aligned}
\<S(q),h\>=\<S(q_1)\ve(q_2),h\>=\<S(q_1),h\>\<q_2,1\>
&=\sum_i\<S(q_1),x^i_1\>\<q_2,x^i_2y_i\>\\
&=\sum_i\<S(q_1),x^i_1\>\<q_2,x^i_2\>\<q_3,y_i\>\\
&=\sum_i\<S(q_1)q_2,x^i\>\<q_2,y_i\>\\
&=\sum_i\<\ve(q_1),x^i\>\<q_2,y_i\>\\
&=\sum_i\ve(q_1)\ve(x^i)\<q_2,y_i\>\\
&=\sum_i\<q,\ve(x^i)y_i\>\\
&=\sum_i\<q,S(x^i_1)x^i_2y_i\>
=\<q,S(h)\>.
\end{aligned}    
\]
}
Summarizing all these, we get

\begin{thm}\label{qcross}
Let $Q$ be a Hopf algebra such that $A\rtimes H$ is a $Q$-module algebra and $A \subset (A\rtimes H)^Q$, where $(A\rtimes H)^Q$ is the invariant subalgebra, that is $Q \in \mathrm{Obj}(\mathrm{C}(A \subset A\rtimes H))$. Moreover, suppose that the extension $A \to A\rtimes H$ is irreducible, i.e., $A' \cap (A\rtimes H)=\bC$. Then there exists a unique pairing 
\begin{equation}
    \<-,-\> : Q \tens H \to \bC, \quad q \tens h \mto \<q,h\>,
\end{equation}
such that
\begin{equation}
    q \cdot (a \rtimes h)=a\rtimes h_1\<q,h_2\>,
\end{equation}
for all $a \in A$, $h \in H$ and $q \in Q$.
%a unique Hopf algebra morphism $\Phi: Q \to \widehat{H}$ such that the following diagram 
% \[
% \begin{tikzcd}
% Q \tens (A\rtimes H) \arrow[dr] \arrow[r, "\Phi \tens 1"] & \widehat{H} \tens (A \rtimes H) \arrow[d] \\
% & A \rtimes H
% \end{tikzcd}
% \] is commutative. Here, $\widehat{H}$ denotes any Hopf algebra dual to $H$, the unadorned arrows denote respective actions.
\end{thm} 

It does not quite make sense to say that $\QGal(A \subset A\rtimes H)=H^*$: as $H$ is not assumed to be finite dimensional, we have stated everything in the language of Hopf pairings. The actual identification of $\QGal(A \subset A\rtimes H)$ requires a bit more work, as described below.

\medskip\noindent

Let $Q$ be a Hopf algebra such that $A\rtimes H$ is a $Q$-module algebra and $A \subset (A\rtimes H)^Q$, where $(A\rtimes H)^Q$ is the invariant subalgebra. Let $\phi$ be the composite $Q \to H^*$ obtained using the isomorphism from Corollary \ref{cor:irr}. Then the computation in {\bf Step 2} above implies that
\begin{equation}
  \phi(q)(hh')=\phi(q_1)(h)\phi(q_2)(h'), \quad \text{ i.e., } \quad m^*\phi(q) \in H^* \tens H^*,
\end{equation}
where $m$ is the multiplication of $H$, $m^*$ is the transpose. This says $\phi(q) \in H^{\circ}$, the finite or Sweedler dual of $H$. Rephrasing Theorem \ref{qcross} in this way, we get

\begin{thm}
  Let $Q$ be a Hopf algebra such that $A\rtimes H$ is a $Q$-module algebra and $A \subset (A\rtimes H)^Q$, where $(A\rtimes H)^Q$ is the invariant subalgebra. Moreover, suppose that the extension $A \to A\rtimes H$ is irreducible, i.e., $A' \cap (A\rtimes H)=\bC$. Then there exists 
  % unique pairing \[\<-,-\> : Q \tens H \to \bC, \quad q \tens h \mto \<q,h\>,\] such that \[q \cdot (a \rtimes h)=a\rtimes h_1\<q,h_2\>,\] for all $a \in A$, $h \in H$ and $q \in Q$.
  a unique Hopf algebra morphism $\phi: Q \to H^{\circ}$ such that the following diagram 
  \begin{equation}
    \begin{tikzcd}
      Q \tens (A\rtimes H) \arrow[dr] \arrow[r, "\phi \tens 1"] & H^{\circ} \tens (A \rtimes H) \arrow[d] \\
      & A \rtimes H
    \end{tikzcd}
  \end{equation}
  is commutative. %Here, $H^{\circ}$ denotes the Sweedler dual to $H$, and the unadorned arrows denote respective actions. 
  Thus $\QGal(A \subset A\rtimes H)=H^{\circ}$.
\end{thm}

We state the above theorem in the von Neumann algebra setting,
\begin{thm}\label{crossed}
Let $M$ be a finite factor admitting an outer action by a finite dimensional Hopf $C^*$-algebra $H$. Then $\QGal(M\subset M\rtimes H)=H^*$.
\end{thm}

We end this subsection with a remark.

\begin{rem}
  We note that our computation did not require the preservation of the canonical trace. Since $H^*$ automatically preserves the canonical trace, the trace-preserving quantum Galois group is also identified with $H^*$.
\end{rem}

\subsection{The invariant subalgebra}\label{subse:inv}
Let $M$ be a finite factor admitting a saturated (\cite{szympeli}) and outer action of a finite dimensional Hopf $C^*$ algebra $H$. Thus $N\equiv M^H\subset M \subset M\rtimes H$ is a Jones triple, i.e., $M\subset M\rtimes H$ is the basic construction of $M^H \subset M$. By \cite{kodakacrossed}, $M$ can be identified with $N\rtimes H^*$. Then, using Theorem \ref{crossed}, without any extra computation, we can conclude

\begin{thm}
The quantum Galois group $\QGal(M^H\subset M)$ of the pair $M^H\subset M$ is $H$.
\end{thm}

Here also, we remark that no trace-preserving assumption was needed. Indeed, this is always the case for an irreducible finite index inclusion as we show now.

\begin{thm}
  Let $N \subset M$ be an irreducible pair of finite factors with $[M:N] < \infty$. Then $\QGal_{\t}(N\subset M)$ is isomorphic to the dual Hopf $C^*$-algebra of \cite[Theorem 4.16]{liu}. 
\end{thm}

\begin{proof}
By \cite[Theorem 4.16]{liu}, there exists an intermediate algebra $N\subset P \subset M$ such that $P \subset M$ is of depth $2$. Since $N\subset M$ is irreducible, $P'\cap M=\bC 1_M$. Thus $P \subset M$ is depth $2$ and irreducible, which means there exists a finite dimensional Hopf $C^*$-algebra $K$ such that $M=P\rtimes K$. Moreover, $N\subset P=M^{K^*}$, yielding a Hopf $\*$-algebra morphism $\Phi : K^* \to \QGal_{\t}(N\subset M)$, intertwining the two actions.

\medskip

On the other hand, by Corollary \ref{min}, $M^{\QGal_{\t}(N\subset M)} \subset M$ is a depth-2 inclusion of finite factors with $[M:M^{\QGal_{\t}(N\subset M)}]$ and the factor $P$ of \cite[Theorem 4.16]{liu} is smallest with this property. Therefore, $M^{K^*}=P\subset M^{\QGal_{\t}(N\subset M)}$.  But the existence of $\Phi : K^* \to \QGal_{\t}(N\subset M)$ from above yields $M^{\QGal_{\t}(N\subset M)} \subset M^{K^*}$ and therefore $M^{\QGal_{\t}(N\subset M)}=P$. Theorem \ref{crossed} applied to the pair $M^{\QGal_{\t}(N\subset M)}=P \subset M=P\rtimes K$, yields $\Psi: \QGal_{\t}(N\subset M) \to K^*$ intertwining the two actions. By uniqueness of universal objects, $\QGal_{\t}(N\subset M) \cong K^*$.
\end{proof}

We then have the following as corollaries. 
\begin{cor}
  Let $N \subset M$ be an irreducible pair of finite factors with $[M:N] < \infty$. Then $\QGal_{\t}(N\subset M)$ is a finite dimensional Hopf $C^*$-algebra.
\end{cor}

\begin{cor}
  Let $N \subset M$ be an irreducible pair of finite factors with $[M:N] < \infty$. Then $\QGal(N\subset M)$ is isomorphic to $\QGal_{\t}(N\subset M)$.
\end{cor}

We ask

\begin{que}
Is there a duality between $\QGal(N \subset M)$ and $\QGal(M \subset M_1)$?
\end{que}
% % This leads us to the following conjecture.
% % 
% % \begin{conj}\label{conjecture}
% % Let $N\subset M$ be a pair of finite factors with $[M:N] < \infty$ and $N'\cap M=\bC 1_M$. Then $\QGal(N\subset M)$ exists without any trace-preserving assumption.
% % \end{conj}
% %

%%%%%%%%%%%%%%%%%%%%%%%%%%%%%%%%%%%%%%%%%%%%%%%%%%%%%%%%%%%%%%%%%%%%%%%%%%%%%
\subsection{Banica's fixed point algebras} \label{subse: banicacomm}
In \cite{banicacommsquare}, Banica obtains a description of commuting squares having $\mathbb{C}$ in the lower left corner, i.e., of the form 
\begin{equation}\label{comm}
\begin{tikzcd}
S \arrow[r,symbol=\subset] & X \\
\mathbb{C} \arrow[u,symbol=\subset] \arrow[r,symbol=\subset] & P \arrow[u,symbol=\subset],
\end{tikzcd}
\end{equation} where $S$ and $P$ are finite dimensional von Neumann algebras. According to this description (see \cite[Theorem 3.1]{banicacommsquare}), the above commuting square is isomorphic to one of the following form 
\begin{equation}
\begin{tikzcd}
S \arrow[r,symbol=\subset] & (P \otimes (S \rtimes \hat{G}))^G \\
\mathbb{C} \arrow[u,symbol=\subset] \arrow[r,symbol=\subset] & P \arrow[u,symbol=\subset],
\end{tikzcd}
\end{equation}
and the vertical subfactor associated to the first commuting square  is of the form $\mathcal{R} \subset (P \otimes (\mathcal{R} \rtimes \hat{G}))^G$. Here $G$ is a compact quantum group of Kac type, $\mathcal{R}$ is the hyperfinite $\mathrm{II}_1$-factor. The action of $\hat{G}$ on $\mathcal{R}$ is outer and is a product-type action built from the action on $S$. The action of $G$ on $P$ is ergodic on the center. Both algebras $(P \otimes (\mathcal{R} \rtimes \hat{G}))^G$ and $(P \otimes (S \rtimes \hat{G}))^G$ are fixed point algebras in the sense of \cite{banicakac} but the outerness of the $\hat{G}$-action on $\mathcal{R}$ enables us to compute explictly the quantum Galois group of the inclusion $\mathcal{R} \subset (P \otimes (\mathcal{R} \rtimes \hat{G}))^G$. We do this in this subsection and start by briefly recalling the constructions in \cite{banicacommsquare}. We also note that although $\mathcal{R}$ and $G$ are infinite-dimensional, the finite-dimensionality of $P$ enables us to use the algebraic smash product rather than the von Neumann crossed product in defining $(P \otimes (\mathcal{R} \rtimes \hat{G}))^G$.

\medskip

Let $H$ be a (not necessarily finite dimensional) Hopf algebra with involutive antipode and $A$ be a $H^{cop}$-module algebra, where `cop' means the comultiplication is reversed. As before, $A \rtimes H^{cop}$ denotes the smash product and it becomes a right $H^{cop}$-comodule algebra (with the coaction say $\pi$). Let $B$ (typically multimatrix) be a right $H$-comodule algebra with coaction $\beta$. Consider the product coaction $\beta \odot \pi$ of $H^{cop}$ on $B \otimes (A\rtimes H^{cop})$:
\begin{equation}\label{eqbanica}
\beta \odot \pi (b \otimes a \rtimes h)=b_0 \otimes a \rtimes h_2 \otimes h_1S(b_1),
\end{equation} where $b \in B$, $a \in A$, $h \in H$ and we used Sweedler notation for the coaction $\beta$. The twist by the antipode makes $B$ into a right $H^{cop,op}$-comodule algebra, where `op' means the multiplication is reversed. Since $B$ and $A\rtimes H^{cop}$ are both $H$-comodules, so is their tensor product. Although it will not, in general, be a comodule {\it algebra}, note nevertheless the following phenomenon.

\begin{lem}
Let $K$ be a Hopf algebra and $M$ and $N$ right comodule algebras over $K^{op}$ and $K$ with coaction $\rho_M$ and $\rho_N$, respectively. Then, the space of coaction invariants $(M\otimes N)^K$ is a subalgebra of $M\otimes N$. 
\end{lem}
\begin{proof}
  Denote by $m_i\otimes n_i$ and $m_j\otimes n_j$ two elements in $(M\otimes N)^K$, with repeated indices implying summation (to lighten the notation). Then we have
\allowdisplaybreaks{
\begin{align*}
  \rho_M \odot \rho_N ((m_i \otimes n_i)(m_j \otimes n_j))&=\rho_M \odot \rho_N (m_im_j \otimes n_in_j)\\
  &=m_{i0}m_{j0} \otimes n_{i0}n_{j0} \otimes m_{j1}m_{i1}n_{i1}n_{j1}\\
  &=(1\otimes 1 \otimes m_{j1})(m_{i0}\otimes n_{i0}\otimes m_{i1}n_{i1})(m_{j0}\otimes n_{j0}\otimes n_{j1})\\
  &=(1\otimes 1 \otimes m_{j1})(m_{i}\otimes n_{i}\otimes 1)(m_{j0}\otimes n_{j0}\otimes n_{j1})\\
  &=(m_{i}\otimes n_{i}\otimes 1)(m_{j0}\otimes n_{j0}\otimes m_{j1}n_{j1})\\
  &=(m_{i}\otimes n_{i}\otimes 1)(m_{j}\otimes n_{j}\otimes 1)\\
  &=m_im_j \otimes n_in_j \otimes 1,
\end{align*}
}which was to be obtained.
\end{proof}

In particular, we have 

\begin{cor}
  The subspace of coaction invariants $(B \otimes A\rtimes H^{cop})^{H^{cop}}$ is an algebra with componentwise multiplication.
\end{cor}

Let us write $C$ for the invariant subalgebra $(B \otimes A\rtimes H^{cop})^{H^{cop}}$. We observe that $A\subset C$ via $a \mapsto 1\otimes a\rtimes 1$. Let $Q$ be a Hopf algebra such that $C$ is a $Q$-module algebra and $A \subset C^Q$, where $C^Q$ is the invariant subalgebra.

\begin{lem}
For $b \in B$, $\beta(b)_{13}=b_0 \otimes 1 \rtimes b_1 \in C$.
\end{lem}

\begin{proof}
It is a simple computation:
\[
\begin{aligned}
\beta \odot \pi (b_0 \otimes 1\rtimes b_1)&=b_0 \otimes 1 \rtimes b_3 \otimes b_2S(b_1)\\
&=b_0 \otimes 1 \rtimes b_2 \otimes \varepsilon(b_1)1\\
&=b_0 \otimes 1\rtimes b_1 \otimes 1.
\end{aligned}    
\] The second equality uses involutivity of the antipode.
\end{proof}

From now on, we assume that $H$ is a compact quantum group algebra with Haar state $\tau$. Then there exists a conditional expectation $E$ of $B\otimes A \rtimes H^{cop}$ onto $C$ given by \[E(b\otimes a \rtimes h)=b_0 \otimes a \rtimes h_2\tau(h_1S(b_1)).\]

\begin{prop}
For $q \in Q$, let $\hat{q} : B \otimes H^{cop} \rightarrow B\otimes A \rtimes H^{cop}$ be the map defined as $\hat{q}(b\otimes h)=q \cdot E(b\otimes 1\rtimes h)$. Then for each $h \in H$,
\begin{equation}\label{eqbanicacent}
  V^{-1}(h_2)\hat{q}(b\otimes h_1) \in A'\cap (B\otimes A\rtimes H^{cop})=B \otimes (A' \cap A\rtimes H^{cop}),
\end{equation} 
where $V^{-1}: H^{cop} \rightarrow B\otimes A\rtimes H^{cop}$ is defined as $h \mapsto 1 \otimes 1\rtimes S(h)$ and $A' \cap A \rtimes H^{cop}$ is the commutant of $A$ in $A \rtimes H^{cop}$.
\end{prop}

\begin{proof}
Let $a \in A$, $b \in B$ and $h \in H$. We compute
\allowdisplaybreaks{
\[
\begin{aligned}
&{}(1\otimes a \rtimes 1)V^{-1}(h_2)\hat{q}(b\otimes h_1)\\
&=(1\otimes a\rtimes 1)(1\otimes 1\rtimes S(h_2))\hat{q}(b\otimes h_1)\\
&=V^{-1}(h_4)(1\otimes 1\rtimes h_3)(1\otimes a\rtimes 1)(1\otimes 1\rtimes S(h_2))\hat{q}(b\otimes h_1)\\
&=V^{-1}(h_3)(1\otimes h_2\cdot a\rtimes 1)\hat{q}(b\otimes h_1)\\
&=V^{-1}(h_3)(1\otimes h_2\cdot a\rtimes 1)q\cdot E(b\otimes 1 \rtimes h_1)\\
&=V^{-1}(h_3)q\cdot E((1\otimes h_2\cdot a\rtimes 1)(b\otimes 1\rtimes h_1))\\
&=V^{-1}(h_3)q\cdot E(b\otimes h_2\cdot a\rtimes h_1)\\
&=V^{-1}(h_3)q\cdot E((b\otimes 1\rtimes h_1)(1\otimes a \rtimes 1))\\
&=V^{-1}(h_2)\hat{q}(b\otimes h_1)(1\otimes a\rtimes 1).
\end{aligned}
\]
}
The fifth equality holds as $E$ is a $(C,C)$-bimodule morphism and $A\subset C$. 
\end{proof}

\begin{cor}\label{banica q form}
Suppose that the $H^{cop}$ action on $A$ is outer. Then for each $q \in Q$, there exists a unique $T_q \in \mathrm{Hom}_{\mathbb{C}}(B\otimes H,B)$ such that 
\begin{equation}
  \hat{q}(b\otimes h)=T_q(b\otimes h_1)\otimes 1\rtimes h_2.
\end{equation}
\end{cor}

\begin{proof}
Since, by hypothesis, $A'\cap A\rtimes H^{cop}=\mathbb{C}$, $A'\cap (B\otimes A\rtimes H^{cop})=B\otimes (A'\cap A\rtimes H^{cop})=B$. Therefore the previous Proposition yields $T_q(b\otimes h) \in B$ such that $V^{-1}(h_2)\hat{q}(h_1)=T_q(b\otimes h)\otimes 1 \rtimes 1$. Then for each $h \in H$ and $b \in B$,
\[\hat{q}(b\otimes h)=(1\otimes 1\rtimes h_2)(T_q(b\otimes h_1)\otimes 1\rtimes 1)=T_q(b\otimes h_1)\otimes 1\rtimes h_2,\]
which was to be obtained. Uniqueness follows applying $\varepsilon$.
\end{proof}

\begin{lem}
For $b \in B$ and $h \in H$, we have $b_0 \otimes 1 \rtimes b_1\tau(hS(b_2))=b_0 \otimes 1\rtimes h_2\tau(h_1S(b_1))$.
\end{lem}

\begin{proof}
    It is enough to prove that for $x,y \in H$, $x_1\tau(yS(x_2))=\tau(y_1S(x))y_2$ and we have \[x_1\tau(yS(x_2))=\tau(y_1S(x_3))y_2S(x_2)x_1=\tau(y_1S(x))y_2,\] from which the result follows.
\end{proof}

Now we fix a $q \in Q$. We observe that an element of $C$ is a finite sum of elements of the form $E(b \otimes a \rtimes h)=b_0 \otimes a \rtimes h_2\tau(h_1S(b_1))$. Since $A \subset C^Q$, the action of $q$ is completely determined by its effect on $b_0 \otimes 1 \rtimes h_2\tau(h_1S(b_1))$, which is to say by the above lemma, by its effect on $b_0 \otimes 1 \rtimes b_1\tau(hS(b_2))$. Moreover, Corollary \ref{banica q form} says that for a fixed $h \in H$, $q$ takes the subspace $E(B \otimes 1 \rtimes h)$ into itself. We summarize all these in the following

\begin{thm}\label{qgalbanica}
$\mathrm{QGal}(A \subset (B\otimes A \rtimes H^{cop})^{H^{cop}})$ is isomorphic with the universal Hopf $\ast$-algebra which acts on 
$E(B \otimes 1\rtimes H^{cop})$ and maps each of the subspaces $E(B \otimes 1\rtimes h)$ into itself, for $h \in H$.  
\end{thm}

We recall that $H$ was assumed to be a compact quantum group algebra. We write $\widehat{H}$ for the subspace of $H^*$ consisting of functionals of the form $\tau(\cdot h)$ for some $h \in H$. It is well-known that $\widehat{H}$ is a Hopf algebra in duality with $H$. We also recall that $\widehat{H^{cop}}$ is $\widehat{H}^{op}$. Therefore we identify $\widehat{H^{cop}}$ with the space consisting of linear functionals of the form $\tau(\cdot S(h))$ for some $h \in H$. There is also the canonical action of $\widehat{H^{cop}}$ on $B$ given by $\omega \rightharpoonup b=b_0\omega(b_1)$, which takes the form $b_0\tau(b_1S(h))=b_0\tau(hS(b_1))$ if $\omega$ is given by $\tau(\cdot S(h))$. Theorem \ref{qgalbanica} can then be rewritten as

\begin{thm}
$\mathrm{QGal}(A \subset (B\otimes A \rtimes H^{cop})^{H^{cop}})$ is isomorphic with the universal Hopf $\ast$-algebra which acts on 
$B$ such that for each $\omega \in \widehat{H^{cop}}$ elements of the form $\omega \rightharpoonup b$ are mapped to elements of the same form.
\end{thm}

To be more explicit, we introduce some notation. 

\begin{lem}\label{centra}
Let $Q$ be a Hopf $\ast$-algebra and $S=S^* \subset Q$ be a subset. Let $\mathrm{C}_{Q}(S)=\{q \in Q \mid qs=sq, \forall s \in S\}$. Then there exists a largest Hopf $\ast$-subalgebra of $Q$ contained in $\mathrm{C}_{Q}(S)$.
\end{lem}

\begin{defn}
We denote by $\mathcal{H}\mathrm{C}_{Q}(S)$ the largest Hopf $\ast$-subalgebra of $Q$ contained in $\mathrm{C}_{Q}(S)$ and call it the Hopf centralizer of $S$.
\end{defn}

\begin{proof}[Proof of Lemma \ref{centra}]
We apply Zorn's lemma to the poset of Hopf $*$-subalgebras $K$ contained in the commutant $\mathrm{C}_{Q}(S)$ (which is clearly nonempty as it contains the trivial one dimensional Hopf algebra) directed by inclusion and we obtain a maximal element, say $K_0$. It is clearly the largest one because given any other element $K$ of the poset, we can consider the $\ast$-algebra generated by $K$ and $K_0$, which is clearly a Hopf $\ast$-subalgebra contained in $\mathrm{C}_{Q}(S)$ and therefore must equal $K_0$.
\end{proof}

For example, if $G$ is a finite group and $H$ is a subgroup, then $\mathcal{H}\mathrm{C}_{\mathbb{C}G}(\mathbb{C}H)$ is $\mathbb{C}C_{G}(H)$, the group algebra of the centralizer. Let us denote the $\widehat{H^{cop}}$ action on $B$ by $\Lambda : \widehat{H^{cop}} \rightarrow \mathrm{End}(B)$ from now on. Then with these notations,

\begin{thm}\label{explicitqgal}
$\mathrm{QGal}(A \subset (B\otimes A \rtimes H^{cop})^{H^{cop}})\cong \mathcal{H}\mathrm{C}_{Q_{aut}(B)}(\Lambda(\widehat{H^{cop}}))$.

% \{q \in Q_{aut}(B) \mid q \Lambda(\omega)=\Lambda(\omega)q  \quad \forall \omega \in \widehat{H^{cop}}\}$.
\end{thm}

\begin{proof}
By the discussion prior to Theorem \ref{qgalbanica}, any $q \in Q_{aut}(B)$ commuting with $\Lambda$, lifts to a linear map of $E(B\otimes 1 \rtimes H^{cop})$. Explicitly, $q \cdot (E(b\otimes a \rtimes h))=(1\otimes a \rtimes 1)q\cdot (b_0\tau(hS(b_1)))$. By universality, there exists a $\Phi : \mathcal{H}\mathrm{C}_{Q_{aut}(B)}(\Lambda(\widehat{H^{cop}})) \rightarrow \mathrm{QGal}(A \subset (B\otimes A \rtimes H^{cop})^{H^{cop}})$ that intertwines the actions.

\medskip
On the other hand, fix a faithful state $\phi$ on $B$ which is invariant under $\Lambda$. Such a $\phi$ exists because $\Lambda$ is a $\ast$-action of a dual compact quantum group algebra on a finite dimensional $\ast$-algebra. Then $\Lambda(\widehat{H^{cop}})$ is a finite dimensional von Neumann algebra acting on the finite dimensional Hilbert space $L^2(B,\phi)$, hence in particular generated by its range projections. By Theorem \ref{qgalbanica}, the $\ast$-action of $\mathrm{QGal}(A \subset (B\otimes A \rtimes H^{cop})^{H^{cop}})$ leaves the image of each $\Lambda(\omega)$ invariant, i.e., for each $q \in \mathrm{QGal}(A \subset (B\otimes A \rtimes H^{cop})^{H^{cop}})$, $P_{\omega}qP_{\omega}=qP_{\omega}$, where $P_{\omega}$ is the range projection of $\Lambda(\omega)$, $\omega \in \widehat{H^{cop}}$. Since the action of $\mathrm{QGal}(A \subset (B\otimes A \rtimes H^{cop})^{H^{cop}})$ is $\ast$-preserving, each $q$ commutes with the range projections $P_{\omega}$ of $\Lambda(\omega)$ for each $\omega \in \widehat{H^{cop}}$. But these range projections $P_{\omega}$ linearly span (the finite dimensional von Neumann algebra) $\Lambda(\widehat{H^{cop}})$, hence each $q$ commutes with the whole of $\Lambda(\widehat{H^{cop}})$. Thus there exists $\Psi : \mathrm{QGal}(A \subset (B\otimes A \rtimes H^{cop})^{H^{cop}}) \rightarrow \mathcal{H}\mathrm{C}_{Q_{aut}(B)}(\Lambda(\widehat{H^{cop}}))$, intertwining the actions. $\Phi \Psi=\mathrm{id}$ follows from universality. And finally, Theorem \ref{qgalbanica} and the discussion prior to it yield $\Psi \Phi=\mathrm{id}$, completing the proof.
\end{proof}

Taking $A$ to be $\mathcal{R}$ and $B$ to be $P$, we obtain the following.

\begin{thm}
The quantum Galois group $\mathrm{QGal}(\mathcal{R} \subset (P\otimes \mathcal{R} \rtimes H^{cop})^{H^{cop}})$ of the vertical subfactor associated to the commuting square \eqref{comm} is isomorphic to $\mathcal{H}\mathrm{C}_{Q_{aut}(P)}(\Lambda(\widehat{H^{cop}}))$. Here $H$ is the dense Hopf $*$-algebra inside the compact quantum group $G$.
\end{thm}

\begin{rem}
We note the following two special cases of Theorem \ref{explicitqgal}.
\begin{enumerate}
  \item Let $H$ be a finite dimensional Hopf $C^*$-algebra and take $B=H^{cop}$ in Theorem \ref{explicitqgal}. Then we recover the conclusion of Theorem 4.1.7.
  \item Let $H$ be $\mathrm{Rep}(SU(2))$, the algebra of representative functions on $SU(2)$, $A$ be the hyperfinite $\mathrm{II}_1$ factor and $B=\mathrm{End}(V(1))=M_2(\mathbb{C})$, $V(1)$ being the 2-dimensional irreducible representation of $SU(2)$. Then $\mathrm{QGal}(A \subset (B\otimes A \rtimes H^{cop})^{H^{cop}})$ is trivial.
\end{enumerate}
\end{rem}

Theorem \ref{explicitqgal} also raises a rather general question.

\begin{que}
When can one have a version of double centralizer theorem for $\mathcal{H}\mathrm{C}_{Q}(S)$?
\end{que}

\section{Epilogue}\label{section-epilogue}
In this section, we describe some further directions that we intend to pursue in the future.

\subsection*{depth-2, reducible inclusions}
We have treated examples coming from depth-2, irreducible inclusions and as we have already mentioned, these are nothing but crossed products by finite dimensional Hopf $C^*$-algebras, see \cite{szymcrossed}. However, in the reducible situation, Hopf algebras don't suffice and \cite{nvcharacteri} have shown that one needs weak Hopf algebras. For details on weak Hopf algebras and their relation to subfactor theory, see also \cites{bnsweak,nikshychkac,nsw,nvgalois,nvsurvey}. We are trying to imitate the computation in Subsection \ref{subseccrossed} when the inclusion arises from a weak Hopf crossed product.

\begin{remu}
On the other hand, the computation of the quantum Galois groups of exotic subfactors such as the Haagerup subfactor as well as infinite depth, discrete subfactors would be very interesting which we plan to do in the future. We are very grateful to the referee for pointing this out to us.
\end{remu}

\subsection*{Intermediate factors and Galois correspondence}
Any theory regarding the notion of a ``Galois group'' should also provide a Galois correspondence, i.e., a (possibly bijective) match between intermediate subalgebras and ``subgroups'' of the ``Galois group''. Indeed, for Kac algebras a similar correspondence was obtained in \cite{izumilongopopa}, while \cite{nvgalois} treats weak Hopf algebras (see also \cite[Corollary E]{jp1} for a result in the same spirit). Our examples show that such a correspondence holds here as well. We are looking at the general situation presently.

According to \cite{nvgalois}, any finite depth subfactor may be realized as an intermediate subfactor of a depth-2 subfactor and using the correspondence, as a crossed product by a coideal subalgebra of a weak Hopf algebra. Thus understanding the weak Hopf crossed product is essential and it may lead to a more direct proof of our existence result which relies on \cite{wangqsymm}.

\subsection*{Some remarks on a categorical construction}
We recall that for a pair of finite factors $N\subset M$ with $[M:N] < \infty$, there exists a canonical tensor category $\mathrm{Bimod}_{N-N}(N\subset M)$ (\cites{ocneanu,nvgalois}) which is generated by the simple subobjects of the bimodules $M_n$, $n \geq -1$, $M_{-1}=N$, $M_0=M$, i.e., the objects are finite direct sums of simple objects of $M^{\tens n}$. The forgetful functor $\mathrm{for} : \mathrm{Bimod}_{N-N}(N\subset M) \to \mathrm{Bimod}_{N-N}$ is then a fibre functor and the pair seems to satisfy the requirements of \cite{tannakaalgebrd} for a reconstruction, yielding a Hopf algebroid over $N$, in the sense \cite{schauenberg}, see also \cite{bsalgbrd}. However, this Hopf algebroid does not keep $N$ fixed for the following simple reason, as observed in \cite{kadison}, page 83: the fixed point subalgebra commutes with $N$. Nevertheless, it would be interesting to see what the reconstruction of \cite{tannakaalgebrd} actually yields. 

\begin{remu}
  We also remark that the weak Hopf algebra, say $H$, obtained in \cite{nvgalois} moreover satisfies $\mathrm{Rep}(H^*)\isom \mathrm{Bimod}_{N-N}(N\subset M)$ as tensor categories.
\end{remu}

% \subsection*{Infinite-index inclusion}
% In \cite{ocneanuinfinite}, the authors considered infinite-index inclusions and a corresponding Galois theory. The finiteness of index was crucial to our existence result. A deeper look is needed for our version of the Galois group in the infinite-index situation.

% \subsection*{Quantum Galois theory in algebra} 
% The reader familiar with cleft extensions (\cite{montgomery}) might have observed that our computation in Subsection \ref{subseccrossed} uses the cleaving map for the Galois extension $A \to A\rtimes H$. It would be interesting to see if the method adapts to more general cleft extensions. We would also like to point out that cleft extensions are special types of Galois extensions, namely a Galois extension with a normal basis. A purely algebraic investigation in this direction would be very welcome.

\subsection*{Quantum Galois theory in the weak context}
The developments in \cites{nvgalois,nvcharacteri,nvsurvey,nsw} already indicate a need to undertake an investigation of universal symmetries in the weak context. The recent preprint \cite{chelsea} initiates such a program in the algebraic world. We hope a similar undertaking in the analytic setting would be beneficial in understanding quantum symmetries.

%%%%%%%%%%%%%%%%%%%%%%%%%%%%%%%%%%%%%%%%%%%%%%%%%%%%%%%%%%%%%%%%%%%%%%%%%%%%%
%%%%%%%%%%%%%%%%%%%%%%%%%%%%%%%%%%%%%%%%%%%%%%%%%%%%%%%%%%%%%%%%%%%%%%%%%%%%%

%\bibliography{bibliography}
%\bibliographystyle{amsalpha}
% \nocite{*}

\end{document}